\documentclass[10pt]{amsart}
\usepackage[latin1]{inputenc}
\usepackage{amsfonts}
\usepackage{amsmath}
\usepackage{amssymb}
\usepackage{amsthm}
\usepackage{fontenc}
\usepackage{graphicx}
\usepackage[margin=1.45in]{geometry}

\usepackage[all]{xy}
\newtheorem{theorem}{Theorem}[section]

\newtheorem{proposition}[theorem]{Proposition}

\newtheorem{corollary}[theorem]{Corollary}
\newtheorem{definition}{Definition}
\newtheorem{remark}{Remark}

\numberwithin{equation}{section}

\DeclareMathOperator{\Trace}{Trace}

\title{Measure Partitions via Fourier Analysis II: Center Transversality in the $L^2$-norm for Complex Hyperplanes} 

\author{Steven Simon}

\begin{document}

\maketitle

\begin{abstract} Applications of harmonic analysis on finite groups were recently introduced to measure partition problems, with a variety of equipartition types by convex fundamental domains obtained as the vanishing of prescribed Fourier transforms. Considering the circle group, we extend this approach to the compact Lie group setting, in which case the annihilation of transforms in the classical Fourier series produces measure transversality similar in spirit to the classical centerpoint theorem of Rado: for any $q\geq 2$, the existence of a complex hyperplane whose surrounding regular $q$-fans are close -- in an $L^2$-sense -- to equipartitioning a given set of measures. The proofs of these results represent the first application of continuous as opposed to finite group actions in the usual equivariant topological reductions prevalent in combinatorial geometry.\\ 
\\
{\footnotesize }
\end{abstract}

\section{Introduction and Main Results} 

	Partition of point collections, bodies, and other measures on Euclidian space have a long history in combinatorial and discrete geometry. On the one hand are the \textit{equipartition} problems, beginning with the well-known ham sandwich theorem: any $d$ finite absolutely continuous measures on $\mathbb{R}^d$ (henceforth to be called \textit{masses}) can be bisected by a single hyperplane. More generally, given $n$ masses $\mu_1,\ldots, \mu_n$ on $\mathbb{R}^d$, one seeks a partition $\mathcal{P}=\{\mathcal{R}_1,\ldots, \mathcal{R}_m\}$ by a fixed type of ``nice" convex regions  -- e.g., those determined by affine independent or pairwise orthogonal hyperplane collections [3, 4, 6, 11, 12, 18, 22, 28], arrangements by more general fans [1, 2, 17], cones on polytopes [15, 26, 31], et cetera -- so that each region contains an equal fraction of each total measure: $\mu_j(\mathcal{R}_i)=\mu_j(\mathbb{R}^d)/m$ for all $1\leq i \leq m$ and all $1\leq j \leq n$. On the other hand are the \textit{center-transversality} questions, in which the goal is to find an affine space $A$ of a specified dimension such that each partition $\mathcal{P}$ with center containing $A$ comes ``sufficiently close'' to equipartitioning each measure. For instance, the classical center-point theorem of Rado [21] claims for any mass $\mu$ on $\mathbb{R}^d$ the existence of  a point $p$ so that $|\mu(H^\pm)-\mu(\mathbb{R}^d)/2|\leq \mu(\mathbb{R}^d)/6$ for the half-spaces $H^{\pm}$ of any hyperplane $H$ passing through $p$, and moreover that the ratio 1/6 is minimal over all masses. 

\subsection{Finite Fourier Analysis and Equipartitions} As shown in [24], measure equipartition problems have a natural reformulation in terms of harmonic analysis on finite groups: if each partition $\mathcal{P}:=\{\mathcal{R}_g\}_{g\in G}$ is naturally indexed by a free group action (e.g., via isometries of $\mathbb{R}^d$), then evaluating measures for each partition determines maps $f_{\mu_j}: G\rightarrow \mathbb{R}, g\mapsto \mu_j(\mathcal{R}_g)$, the Fourier expansion  \begin{equation} \mu_j(\mathcal{R}_g)= \sum_\sigma d_\sigma \Trace(c_{j,\sigma} \sigma_g) \end{equation} of which is given explicitly in terms of the group's irreducible unitary representations $\sigma: G\rightarrow U(d_\sigma)$ and their matrix-valued  transforms [7]
	 \begin{equation} c_{j,\sigma}=\frac{1}{|G|}\sum_{g\in G} \mu_j(\mathcal{R}_g)\sigma_g^{-1} \in M(\mathbb{C}, d_\sigma) \end{equation} 

A variety of equipartition types could then be obtained as the vanishing of judiciously chosen $c_{j,\sigma}$, including full equipartitions when all transforms but that given by the trivial representation were annihilated. Many previously considered problems could be seen as special cases of this viewpoint, including the extensively studied Gr\"unbaum and ``generalized Makeev" hyperplane problems  mentioned above as $\mathbb{Z}_2^k$ examples, including the ham sandwich theorem when $k=1$. 

\subsection{Compact Groups and $L^2$ Center-Transversality} Considering the circle group $S^1$, we extend this harmonic analysis approach to the compact Lie group setting, in which case the vanishing of Fourier transforms will produce an ``average" measure center transversality, described precisely in terms of the $L^2$-norm on the group. First, we set some notation. 
	
\begin{definition} Let $q\geq 2$. A \textit{complex regular $q$-fan} $F_q$ in $\mathbb{C}^d$ is the union of $q$ half-hyperplanes, centered about a complex hyperplane $H_{\mathbb{C}}$, whose successive dihedral angles are all equal to $2\pi/q$. The resulting closed regions are called regular ($q$-)sectors. \end{definition}  
	
	The regular sectors of all complex regular $q$-fans $F_q$ centered about a fixed complex hyperplane are  parametrized by the free rotational $S^1$-action, so as before one has a natural class of convex decompositions $\mathcal{D}_q(S^1)=\{\mathcal{S}_{q,\lambda}\}_{\lambda\in S^1}$ and resulting maps $f_{\mu_j}: S^1\rightarrow \mathbb{R},\, \lambda \mapsto \mu_j(\mathcal{S}_{q,\lambda}),$ whose Fourier series \begin{equation} \mu_j(\mathcal{S}_{q,\lambda}) =\sum_{m\in\mathbb{Z}} c_{j,m} \lambda ^m \end{equation} 
converge uniformly because the measures are continuous. Technically speaking, our main result guarantees the vanishing of arbitrarily prescribed Fourier coefficients \begin{equation} c_{j,m}=\int_{S^1}\mu_j(\mathcal{S}_{q,\lambda})\lambda^{-m}d\lambda\in \mathbb{C}, \end{equation} and hence their conjugates $\overline{c_{j,m}}=c_{j,-m}$, whenever $m\neq 0$:

\begin{theorem} Let $q\geq 2$. For any $d$ masses on $\mathbb{C}^d$ and any $m_1,\ldots, m_d>0$, there exists a complex hyperplane whose surrounding regular $q$-sectors  $\{\mathcal{S}_{q,\lambda}\}_{\lambda \in S^1}$ satisfy $c_{j,m_j}=c_{j,-m_j}=0$ in (1.3) for each $1\leq j \leq d$. \end{theorem}

	For the finite subgroups $\mathbb{Z}_q$ of the circle, annihilating coefficients showed that any $d$ masses on $\mathbb{C}^{d(q-r)}$ could be simultaneously equipartitioned by each of the $r$ regular $p$-fans composing some complex regular $q=pr$-fan $F_q$, provided $p$ was prime (see [24]). Passing to the full Lie group, for arbitrary $q\geq 2$  we have a complex hyperplane whose surrounding regular $q$-fans are all close to equipartitioning a given set of masses, provided the $L^2$-norm for functions on the circle is used. For instance, annihilating each first Fourier coefficient will show that 
	
\begin{corollary}  For any $d$ masses $\mu_1,\ldots, \mu_d$ on $\mathbb{C}^d$ and any $q\geq 2$, there exists a collection of regular $q$-sectors centered about a complex hyperplane such that \begin{equation} \|\mu_j(\mathcal{S}_{q, \lambda}) - \mu_j(\mathbb{C}^d)/q\|_2 \leq \sqrt{\frac{1}{3} -  \frac{2}{\pi^2} - \frac{1}{3q^2}}\, \mu_j(\mathbb{C}^d) \end{equation}  for each $1\leq j \leq d$. 
\end{corollary}

One comes arbitrarily close to equipartitioning in sufficiently high dimensions by annihilating further coefficients:  

\begin{corollary} For any $d$ masses $\mu_1,\ldots, \mu_d$ on $\mathbb{C}^{dn}$ and any integer $q\geq 2$, there exists a collection of regular $q$-sectors centered about a complex hyperplane such that \begin{equation}  \|\mu_j(\mathcal{S}_{q,\lambda}) - \mu_j(\mathbb{C}^{dn})/q\|_2< \frac{\kappa}{\sqrt n} \, \mu_j(\mathbb{C}^{dn}) \end{equation} for each $1\leq j \leq d$, where $\kappa$ is a constant independent of $d, n,$ and $q$. 
\end{corollary} 

Corollary 1.2 is of particular interest when $d=1$ and $q=2$. In this case a complex hyperplane is a point in $\mathbb{R}^2$ and a regular 2-fan is a line, so (1.5) is a $L^2$-analogue in the plane of the centerpoint theorem above, i.e., the existence of a point $c$ whose surrounding regular 2-sectors satisfy the uniform bound $\|\mu(\mathcal{S}_{2,\lambda}) - \mu(\mathbb{R}^2)\|_\infty=\max_{\lambda\in S^1} |\mu(\mathcal{S}_{2,\lambda}) - \mu(\mathbb{R}^2)| \leq \mu(\mathbb{R}^2)/6$. This $1/6$ is smaller than that of Corollary 1.2, but Proposition 2.1 below will show that our $L^2$ bounds are smaller than those in the $L^\infty$ norm if $q\neq 2, 4, 5, 6$. We also observe that even though the center-transversality obtainable via Theorem 1.1. is in general most naturally stated in terms of the $L^2$-norm, some $L^\infty$ estimates can also be obtained via transform annihilation if further assumptions on the measures are made, e.g., the ``rotational acceleration" criteria of Remark 1 below.  

\subsection{Equivariant Topological Underpinnings} Measure equipartitions are ordinarily obtained topologically, with the configuration-space/test-map paradigm, applied throughout discrete geometry and combinatorics, the main tool for the problem's reduction to an equivariant framework (see, e.g., [19, 29, 30] for an introduction). Although suggested by Sarkaria [23] for their potential use in the topological Tverberg conjecture, our proof of Theorem 1.1 given in Section 3 represents the first use of continuous group actions in this scheme. As in the finite cases, owing to the group symmetry on each decomposition, the desired partition is shown to be equivalent to the zero of a certain equivariant map, guaranteed here by a simple degree argument. A variety of powerful cohomological techniques are employed in this context more generally (e.g., the ideal-valued index theory of [10], characteristic and more general obstruction classes, et cetera) and it should be noted that the specific use of the circle group offers a technical advantage. Namely, while torsion in finite group cohomology placed restrictions on the partition types  (i.e., transform annihilation) obtainable by these methods, including full equipartitions only when the number of regions is a prime power (see, e.g., [4, 5, 15, 16]), there are only trivial conditions on transform annihilation here because the group cohomology is now torsion free. In particular, our center-transversality holds for partitions by an arbitrary number of sectors.

\section{Transversality in the $L^\infty$-norm and Proof of Corollaries}

Before proving the above corollaries, we provide a comparison of the $d=1$ case of Corollary 1.2 to the corresponding uniform bounds. 

\begin{definition} Let $\varepsilon_\infty(q)$ denote the minimum $\varepsilon$ such that for any mass $\mu$ on $\mathbb{R}^2$, there exists some $\{\mathcal{S}_{q,\lambda}\}_{\lambda\in S^1}$ for which $\|\mu(\mathcal{S}_{q,\lambda})-\mu(\mathbb{R}^2)/q\ \|_\infty  \leq \varepsilon \mu(\mathbb{R}^2)$, and let $\varepsilon_2(q)$ be the analogous minimum  in the $L^2$-norm. \end{definition} 

The original centerpoint theorem is $\varepsilon_\infty(2)=1/6$, while for $q\geq 3$ a continuous extension of the ``wedge centerpoint"  theorems of [9] given for point collections shows that 

 \begin{proposition} (a) $\varepsilon_\infty(q)=\max \{1/q, (q-2)/2q\}$ for $q\geq 3$ and\\ (b) $\varepsilon_2(q)\leq  \sqrt{\frac{1}{3} -  \frac{2}{\pi^2} - \frac{1}{3q^2}}<\varepsilon_\infty(q)$ for $q=3$ and all $q>6$. \end{proposition} 
 
 \begin{proof}  To show that $\epsilon_\infty(q)\leq \max \{1/q, (q-2)/2q\}$, it suffices as usual (see e.g., [26, 27]) to assume that the density function of $\mu=h\,dm$ is $C^\infty$ with connected compact support. As in [9], the key ingredient is demonstrating a regular $6$-fan whose three lines all bisect $\mu$, and this follows from the intermediate value theorem: as in the original proof [25] of the ham sandwich theorem, each $x\in S^1$ determines a unique $t(x)\in\mathbb{R}$ such that the line $L(x):=\{u\mid \langle u, x \rangle =t(x)\}$  bisects $\mu$. Letting $L_k(x):=L(\exp(\pi ik/3)x)$ for each $0\leq k\leq 2$ and denoting their intersections by $p_k(x)=L_k(x)\cap L_{k+1}(x)$, the three lines form a regular 6-fan provided $p_1(x)\in L_0(x)$. If $f(x)= \langle p_1(x), x\rangle -t(x)x$, it follows that $f(-x)=-f(x)$ for all $x\in S^1$, and therefore some $\cup_{k=0}^2L_k(x)$ must be a regular 6-fan. Letting $c$ be the center of such a $F_6$, one easily sees that $0\leq \mu(\mathcal{S}_{q,\lambda}) \leq \mu(\mathbb{R}^2)/2$ for each of the regular $q$-sectors centered at $c$, and hence that $\|\mu(\mathcal{S}_{q,\lambda})-\mu(\mathbb{R}^2)/q\|_\infty \leq \max \{1/q, (q-2)/2q\}\mu(\mathbb{R}^2)$. 
 
 	That $\epsilon_\infty(q) \geq\max\{1/q,(q-2)/2q\}$ follows as in Lemmas 1 and 5 of [9]: For large $r>0$, consider $2n$ points $P$ on the real line which are separated into two collections $P_1\subset [-r-1,-r]$ and $P_2\subset [r,r+1]$ of $n$ points each. For small $\delta>0$, let $\mu_{r,\delta}$ consist of the $2n$ disjoint disks of radius $\delta$ whose centers are in $P$. Assuming that $r$ is sufficiently large and $\delta$ is sufficiently small,  the distance between $c$ and $P_1$ or $c$ and $P_2$  is at least $r$ for any point $c\in\mathbb{R}^2$, so any $\{\mathcal{S}_{q,\lambda}\}_{\lambda \in S^1}$ has some regular $q$-sector containing all the disks with centers in $P_1$ or all the disks with centers in $P_2$. On the other hand, one of these sectors has measure at most $\pi\delta^2/q$. Thus $\|\mu_{r,\delta}(\mathcal{S}_{q,\lambda}) - \mu_{r,\delta}(\mathbb{R}^2)\|_\infty \geq \max\{1/q -1/2nq, (q-2)/2q\}\mu_{r,\delta}(\mathbb{R}^2)$ for any point in $\mathbb{R}^2$, so $\varepsilon_\infty(q)=\max\{1/q, (q-2)/2q\}$, and therefore $\varepsilon_2(q)\leq \sqrt{\frac{1}{3} -  \frac{2}{\pi^2} - \frac{1}{3q^2}}<\varepsilon_\infty(q)$ for $q=3$ and $q>6$ by Corollary 1.2. 
 \end{proof} 
 
 \subsection{Proof of Corollaries 1.2--1.3} 
 
\begin{proof} Let $q\geq 2$, and suppose that $\mu_1=h_1\,dm,\ldots, \mu_d=h_d\,dm$, $h_j:\mathbb{C}^{dn}\rightarrow [0,\infty)$. For any $\mathcal{D}_q(S^1)=\{\mathcal{S}_{q,\lambda}\}_{\lambda\in S^1}$, it is easily seen from substitution that (i) $c_{j,m}=0$ for all $m\in q\mathbb{Z}_+$ and that (ii) $c_{j,0}=\mu_j(\mathbb{C}^{dn})/q$. By Theorem 1.1, there exists some complex hyperplane $H_{\mathbb{C}}$ with  $c_{j,\pm m}=0$ for each $1\leq j \leq d$ and each $1\leq m \leq n$, so
\begin{equation} \| \mu_j(\mathcal{S}_{q,\lambda}) - \mu_j(\mathbb{C}^{dn})/q\|_2^2 = 2\sum_{\stackrel{m>n}{m \notin q\mathbb{Z}_+}} |c_{j,m}|^2 \end{equation} by the Parseval identity. Letting $f_{\mu_j}(\lambda)=\mu_j(\mathcal{S}_{q,\lambda})$ and identifying $\lambda$ with $\theta\in [-\pi,\pi]$, one has $|c_{j,m}| \leq \frac{V_j}{m}$, where  $V_j=\frac{1}{2\pi} \int_{-\pi}^\pi |f_{\mu_j}'(\theta)|\,d\theta$ is the total variation, and the latter is trivially bounded by $\mu_j(\mathbb{C}^{dn})/\pi$: After changing to polar coordinates and a possible translation, we may assume $f_{\mu_j}(\theta)= \int_\theta^{\theta +2\pi/q} \int_0^\infty h_j(r,\phi)r\,dr \, d\phi$, so $f_{\mu_j}'(\theta)=\int_0^\infty [h_j(r,\theta+2\pi/q)-h_j(r,\theta)]r\,dr$, and therefore $V_j\leq \frac{1}{2\pi} \int_{-\pi}^\pi \int_0^\infty |h_j(r,\theta+2\pi/q)-h_j(r,\theta)|r\,dr\,d\theta  \leq \frac{1}{\pi} \int_{-\pi}^\pi \int_0^\infty h_j(r,\theta)r\,dr\,d\theta=\mu_j(\mathbb{C}^{dn})/\pi$. One then has $\|\mu_j(\mathcal{S}_{q,\lambda}) - \mu_j(\mathbb{C}^{dn})/q)\|_2^2 \leq  \frac{2\mu_j(\mathbb{C}^{dn})^2}{\pi^2}\sum_{m>n, m\notin q\mathbb{Z}_+}\frac{1}{m^2}<\frac{2\mu_j(\mathbb{C}^{dn})^2}{n\pi^2}$. This proves Corollary 1.3, while the explicit value $\sum_{m>1, m\notin q\mathbb{Z}_+} m^{-2}=(q^2-1)\pi^2/6q^2 -1$ gives (1.5).\end{proof}
			
\begin{remark} Using similar reasoning, it is easy to see that uniform bounds via transform annihilation also hold given further smoothness assumptions on the measures. For instance, if the density functions of the $\mu= h\,dm$ are $C^1$, estimates can be given control on the ``rotational acceleration"  $[\mu(\mathcal{S}_{q,\theta})]''$ of the measures about the various complex hyperplanes. In particular, for any such $\mu_1=h_1\,dm,\ldots, \mu_d=h_d \,dm$ on $\mathbb{C}^d$, there exists a complex hyperplane such that \begin{equation*} |\mu_j(\mathcal{S}_q) - \mu_j(\mathbb{C}^d)/q| \leq 2A_j  \left(\frac{(q^2-1)\pi^2}{6q^2} -1\right) \end{equation*}  for each sector $\mathcal{S}_q$ and each $1\leq j \leq d$, where $A_j=\sup_{H_{\mathbb{C}}} \frac{1}{2\pi} \int_{-\pi}^\pi \left| [\mu_j(\mathcal{S}_{q,\theta})]'' \right |\,d\theta$. \end{remark}

\section{Proof of Theorem 1.1}

	Finally, we prove Theorem 1.1 via the configuration-space/test-map paradigm mentioned in the introduction. For the \textit{configuration-space}, to each $x=(\mathbf{a},b)\in S(\mathbb{C}^{d+1})$, $\|\mathbf{a}\|^2 + |b|^2=1$, we define the sets 	\begin{equation} \mathcal{S}_{q,\lambda}(x)=\{\mathbf{u}\in\mathbb{C}^d\mid \ \langle \mathbf{u},\mathbf{a} \rangle_{\mathbb{C}} + \bar{b}=\lambda v,\, |\arg v|\leq \pi/q\} \end{equation} for each $\lambda\in S^1$, where $\langle\mathbf{u},\mathbf{a}\rangle=\sum_{i=1}^du_i\bar{a}_i$ denotes the standard Hermitian form on $\mathbb{C}^d$.  Note that each collection $\mathcal{D}_q(S^1)$ of $\mathbb{C}^d$ is uniquely realized by the $S^1$-orbit of some $x\notin \mathbf{0}\times S^1$, and moreover that the $S^1$-action on the orbit corresponds precisely to the rotations of the sectors about their centering complex hyperplane $H_{\mathbb{C}}(x)=\{\mathbf{u}\mid \langle \mathbf{u}, \mathbf{a} \rangle_{\mathbb{C}}+\bar{b}=0\}$. On the other hand, the sets $\mathcal{S}_\lambda(\mathbf{0},b)$ are all of $\mathbb{C}^d$ if $\arg \lambda \in I_b:=[-\pi/q - \arg b, \pi/q- \arg b]$ and are empty otherwise.

	 Now suppose that $\mu_1,\ldots, \mu_d$ are masses on $\mathbb{C}^d$. Given $m_1,\ldots, m_d>0$, consider the representation $\rho=\oplus_{j=1}^d \chi_{m_j}: S^1\rightarrow U(d)$, the direct sum of the power maps $\chi_{m_j}(\lambda)=\lambda^{m_j}$. Evaluating the Fourier coefficients ranging over all $\mathcal{D}_q(S^1)=\{\mathcal{S}_{q,\lambda}\}_{\lambda\in S^1}$ can then be extended to a continuous \textit{test-map} $\mathcal{F}: S(\mathbb{C}^{d+1}) \rightarrow \mathbb{C}^d$ defined on the full sphere by $\mathcal{F}(x) = (c_{1,m_1}(x),\ldots, c_{d,m_d}(x))$, where
		
		\begin{equation} c_{j,m_j}(x)= \int_{S^1} \mu_j(\mathcal{S}_{q,\lambda}(x))\lambda^{-m_j}\,d\lambda \end{equation} 

	Thus Theorem 1.1 will be proven if we can find some $x\notin \mathbf{0}\times S^1$ for which $\mathcal{F}(x)=0$. Crucially, since $\mathcal{S}_{q,\lambda_1}(\lambda_2x) = \mathcal{S}_{q,\lambda_1\lambda_2}(x)$ for all $\lambda_1,\lambda_2\in S^1$ for any of the sets (3.1), this map is equivariant with respect to the standard action on $S(\mathbb{C}^{d+1})$ and the linear action on $\mathbb{C}^d$ given by $\rho$. Continuity of $\mathcal{F}$ is demonstrated in an essentially standard fashion below, so a zero of this map can be guaranteed from the degree calculation of Proposition 3.1. As $c_{j,m_j}(\mathbf{0},b)= \int_{\arg \lambda \in I_b} \mu_j(\mathbb{C}^d) \lambda^{-m_j}\,d\lambda \neq 0$, the zero cannot lie in $\textbf{0}\times S^1$, as desired. 

	For continuity, it suffices to consider a single mass and a single $\chi_m: S^1\rightarrow U(1)$.  To show that $\int_{S^1}\mu(\mathcal{S}_{q,\lambda}(x_n))\lambda^{-m}\,d\lambda\to \int_{S^1}\mu(\mathcal{S}_{q,\lambda}(x))\lambda^{-m}d\lambda$ if $x_n=(\mathbf{a}_n,b_n) \to x=(\mathbf{a},b)$, one notes that the set $\partial \mathcal{S}_{q,\lambda}(x):= \{\mathbf{u}\in\mathbb{C}^d\mid \ \langle \mathbf{u},\mathbf{a} \rangle_{\mathbb{C}} + \bar{b}=r\lambda e^{\pm \pi i/q},\,r\geq 0\}$ is half of a real hyperplane when $x\notin \mathbf{0}\times S^1$, while for $x=(\mathbf{0},b)$ it is all of $\mathbb{C}^d$ if $\arg\lambda=\pm \pi/q - \arg b$ and is empty otherwise. Thus $\mu(\partial \mathcal{S}_{q,\lambda}(x))=0$ for all but at most two values of $\lambda$, so as in the proof of the ham sandwich theorem given in [19], it follows from the dominated convergence theorem that $\mu(\mathcal{S}_{q,\lambda}(x_n))\lambda^{-j}\rightarrow \mu(\mathcal{S}_{q,\lambda}(x))\lambda^{-j}$, almost everywhere. Thus $\mathcal{F}(x_n)\rightarrow \mathcal{F}(x)$ again by dominated convergence, since  $|\mu(\mathcal{S}_{q,\lambda}(x_n))\lambda^{-j}|\leq \mu(\mathbb{C}^d)<\infty$. 

\begin{proposition} If $S^1$ acts on $S(\mathbb{C}^{d+1})=S^{2d+1}$ by the standard action and linearly on $\mathbb{C}^d$ via $\rho=\oplus_{j=1}^d \chi_{m_j}:S^1\rightarrow U(d),$ then any continuous $S^1$-equivariant map $h: S^{2d+1} \rightarrow \mathbb{C}^d$ has a zero if $m_1,\ldots, m_d\neq 0$. 
\end{proposition}

\begin{proof} If $h$ were not to have a zero, the usual argument shows that composition of $k(x):=h(x)/\|h(x)\|$ with the nullhomotopic inclusion $S^{2d-1}\hookrightarrow S^{2d+1}$ produces a degree zero $S^1$-equivariant map $f: S^{2d-1}\rightarrow S^{2d-1}$. This is a contradiction, because $f$ has degree $m=m_1\cdots m_d$ \textit{mod} $q$ for any integer $q\geq 2$  (see, e.g., Proposition 4.12 of [8]), hence overall degree $m\neq 0$. \end{proof} 

\subsection{A Cohomology Perspective} As many of the results in combinatorial and discrete geometry obtained by topological reduction rely on equivariant cohomological methods, we conclude with a discussion of how the lack of restrictions on transform annihilation of Theorem 1.1, as opposed to that for finite groups in the CS/TM-scheme, can be seen at the level of group cohomology. Recall that for each group $G$, there is a classifying space $BG$, unique up to homotopy, which is the quotient of a contractible space $EG$ on which $G$ acts freely: $G\hookrightarrow EG\rightarrow BG=EG/G$ (see, e.g., [14]), and one can define the cohomology of the group $G$ as $H^*(BG;\mathbb{Z})$. For the circle group, one can take $ES^1=S^\infty=\cup_d S(\mathbb{C}^{d+1})$ with the standard $S^1$ action, so that $BS^1=S^\infty/S^1=\cup_d\mathbb{C}P^d=\mathbb{C}P^\infty$ is infinite dimensional complex projective space. It is a basic fact that $H^*(\mathbb{C}P^\infty;\mathbb{Z})$ is simply the polynomial ring $\mathbb{Z}[a]$. 

	A typical means of showing that the equivariant map $h: S^{2d+1}\rightarrow \mathbb{C}^d$ vanishes is to quotient the trivial bundle $S^{2d+1}\times\mathbb{C}^d$ by the diagonal $S^1$-action and demonstrate that the resulting complex vector bundle $\mathbb{C}^d \hookrightarrow E:= S^{2d+1} \times_{S^1} \mathbb{C}^d \rightarrow \mathbb{C}P^d$ does not admit a non-vanishing section, and therefore that the particular section $s: \mathbb{C}P^d\rightarrow E$ induced from $x\mapsto (x,h(x))$ must have a zero. To that end, one calculates the bundle's top Chern class $c_d(E)\in H^{2d}(\mathbb{C}P^d;\mathbb{Z})$ to be non-zero (see, e.g., [14, 20]).  From the viewpoint of classifying spaces, $E$ is the pullback under inclusion $i:\mathbb{C}P^d\hookrightarrow \mathbb{C}P^\infty$ of the bundle $E_\rho:=ES^1\times_{S^1} \mathbb{C}^d$ given by the Borel construction, so that $c_d(E)=i^*(c_d(E_\rho))$ by naturality. As $i^*:H^*(\mathbb{C}P^\infty;\mathbb{Z})\rightarrow H^*(\mathbb{C}P^d;\mathbb{Z})$ is an isomorphism in dimensions $n\leq 2d$ (see, e.g., [13]), one is ultimately reduced to the computation $c_d(E_\rho)=m_1\cdots m_d\, a^d\neq 0$ in the group cohomology $H^*(BS^1;\mathbb{Z})$. 
	
\section{Acknowledgments} 

The author is grateful for the suggestions and comments of the anonymous reviewer, which were helpful in improving the exposition of this paper. This research was partially supported by ERC advanced grant 32094 during visits with Gil Kalai at the Hebrew University of Jerusalem, whom the author thanks for many beneficial conversations. 

\bibliographystyle{plain}

\end{document}